\documentclass[12pt]{amsart}

\usepackage[USenglish]{babel}
\usepackage[utf8]{inputenc}
\usepackage{amsfonts}
\usepackage{amssymb}
\usepackage{amsthm}
\usepackage{todonotes}
\usepackage{comment}
\usepackage{tikz}
\usepackage{ upgreek }
\usepackage{dsfont}
\usepackage{epic}
\usepackage{indentfirst}
\usepackage[centering, includeheadfoot, hmargin=1in, vmargin=0.7in,
  headheight=30.4pt]{geometry}
\usepackage[pdftitle={},
	    pdfauthor={Carsten Borntr\"ager, Matthias Nickel},
	    linktocpage=true
	    ]{hyperref}

\DeclareMathOperator{\vol}{vol}

\DeclareMathOperator{\BigC}{Big}
\DeclareMathOperator{\End}{End}

\newtheorem{theorem}{Theorem}
\newtheorem{lemma}[theorem]{Lemma}
\newtheorem{prop}[theorem]{Proposition}
\newtheorem{cor}[theorem]{Corollary}
\newtheorem{question}[theorem]{Question}
\newtheorem{defin}[theorem]{Definition}
\theoremstyle{remark}
\newtheorem{rem}[theorem]{Remark}

\newcommand{\BC}{{\mathbb{C}}}
\newcommand{\BN}{{\mathbb{N}}}
\newcommand{\BP}{{\mathbb{P}}}
\newcommand{\BQ}{{\mathbb{Q}}}
\newcommand{\BR}{{\mathbb{R}}}
\newcommand{\BV}{{\mathbb{V}}}
\newcommand{\BZ}{{\mathbb{Z}}}

\newcommand{\mcO}{{\mathcal{O}}}
\newcommand{\mcE}{{\mathcal{E}}}
\newcommand{\md}{{\mathrm{d}}}

\begin{document}

\setlength{\parindent}{2ex}

\title[Algebraic Volumes of Divisors]{Algebraic Volumes of Divisors}

\author[C. Borntr\"ager, M. Nickel]{Carsten Borntr\"ager, Matthias Nickel}
\address{Philipps-Universit\"at Marburg, Institut f\"ur Mathematik, Hans-Meerwein-Straße, D-35032 Marburg, Germany}
\email{borntraegerc@mathematik.uni-marburg.de}
\address{Goethe-Universit\"at, Institut f\"ur Mathematik, Robert-Mayer-Str. 6--8, D-60325 Frankfurt am Main, Germany}
\email{nickel@math.uni-frankfurt.de}

\date{\today}

\begin{abstract}
The volume of a Cartier divisor on a projective variety is a nonnegative real number that measures the asymptotic growth of sections of multiples of the divisor. It is known that the set of these numbers is countable and has the structure of a multiplicative semigroup. At the same time it still remains unknown which nonnegative real algebraic numbers arise as volumes of Cartier divisors on some variety.

Here we extend a construction first used by Cutkosky, and use the theory of real multiplication on abelian varieties to obtain a large class of examples of algebraic volumes. We also show that $\pi$ arises as a volume.
\end{abstract}

\maketitle

\section{Introduction}
\noindent The main purpose of this paper is to realize certain irrational algebraic numbers as volumes of Cartier divisors. 
The volume of a Cartier divisor $D$ on a projective complex variety $X$ measures the asymptotic rate of growth of global sections of its multiples.
If $\dim(X)=d$, then  \[ \vol_X(D)= \limsup_{n \rightarrow \infty} \frac{h^0(\mcO_X(n D))}{n^d/d!} \ . \]
For nef divisors the asymptotic Riemann--Roch theorem yields $\vol_X(D)=(D^d)$.

The volume was first used implicitly by Cutkosky \cite{cutkosky86} to study the existence of Zariski decomposition on higher dimensional varieties and has evolved into a fundamental invariant of line bundles in projective geometry. In \cite{cutkosky86} Cutkosky shows that there is an effective divisor on a certain threefold which has irrational volume, hence can not have a Zariski decomposition even after a birational modification.

The volume function itself enjoys many interesting formal properties: it depends only
on the numerical equivalence class of the divisor and can be extended to a continuous function on the real N\'{e}ron--Severi space $N_\BR^1(X)$.

We will be primarily interested in the set of volumes \[\BV:=\{ a \in \BR_+ \mid a = \vol_X(D) \text{ for some pair } (X,D) \text{ with $D$ an integral Cartier divisor on $X$} \} ,\]
which is known to contain $\BQ_+$ and is a multiplicative semigroup by the Künneth formula \cite[Remark 2.42]{Hab}. \\

Even though the volume function is locally piecewise polynomial with rational coefficients on surfaces \cite{BKS} and in the case of finite generation \cite{KKL}, it can can be irrational or even transcendent in the absence of finite generation phenomena \cite[3.3]{BKS}, \cite[3]{Volumefunctions}.
Furthermore in \cite{Volumefunctions} the authors verify that there are only countably many volume functions for all irreducible projective varieties, in particular, $\BV$ is countable. They also show the existence of a fourfold where the volume function is given by a transcendental formula on an open subset of the big cone $\BigC(X)_{\BR}$, which illustrates that the behaviour of $\BV$ is rather mysterious.

Any volume arises as the Lebesgue volume of a convex set, called the Newton--Okounkov body associated to the divisor which was introduced in \cite{KK} and in \cite{LMconvex}. This leads the authors of \cite{Volumefunctions} to ask if the volume of an integral Cartier divisor is always a period in the sense of \cite{KontZagier} which for example would be true if for each Cartier divisor there existed a Newton-Okounkov body whose boundary is given by the zero sets of finitely many polynomials with rational coefficients. 

Here we try to understand the connection between $\BV$ and $\overline{\BQ}$, and will show that $\BV$ contains a large class of positive real algebraic numbers.
		
\begin{theorem}\label{thm:haupt}
For every totally real Galois number field $K$ there exists a smooth projective variety $X$ and
a divisor $D$ on $X$ such that $\vol_X(D)$ is a primitive element of $K$. 
\end{theorem}

The idea of the proof is to use Cutkosky's construction in the case of projective bundles over abelian varieties with real multiplication.  

Going in the other direction, we establish the following result.

\begin{prop}
$\pi$ is an element of $\BV$.
\end{prop}

\subsection{Acknowledgments}

The authors would like to thank Thomas Bauer, the thesis advisor of the first author, Matteo Costantini, Julian Großmann, Victor Lozovanu, Martin Lüdtke, Jakob Stix, and Jürgen Wolfart for many helpful discussions and suggestions. The second author would also like to thank the Centre de Recerca Matem\`atica and the Universit\'e de Caen for their hospitality. Finally the authors would like to express their gratitude to Alex Küronya, the thesis advisor of the second author, for suggesting the problem and the cooperation, his support and many useful comments.

\section{Preliminaries}

\subsection{Volumes on projective bundles}
Let $V$ be an irreducible projective variety of dimension $v$, $A_0,\dots,A_r$ Cartier divisors on $V$. We set \[\mcE = \mcO_V(A_0) \oplus \dots \oplus \mcO_V(A_r)\] and consider the 
projective bundle \[ X = \BP(\mcE)\,,  \]
which is an irreducible projective variety of dimension $d=v+r$.

One then has the following Lemma, which is probably known to experts. For lack of a suitable reference, we provide a proof. 

\begin{lemma}
\label{int}
\begin{equation*} \vol_X \left( \mcO_X(1) \right)= \frac{d!}{v!} \int_{\substack{\lambda_0+\dots+\lambda_r=1\\\lambda_i \geq 0}} \vol_V(\lambda_0 A_0 + \dots + \lambda_r A_r)\, \md \lambda_1 \dots \md \lambda_r.
\end{equation*}
\end{lemma}

\begin{proof}
Let $\pi: X \rightarrow V$ be the projection. The fact that $\pi_* \mcO_X(1) = \mcE$ yields
\[H^0(X, \mcO_X(k))= \bigoplus_{a_0+\dots+a_r=k} H^0(V,\mcO_V(a_0 A_0+ \dots + a_r A_r)) \,.\]
We then obtain
\begin{align}
 \limsup_k \frac{ h^0(X,\mcO_X(k))}{k^d/d!} &=\limsup_k \sum_{a_0+\dots+a_r=k} \frac{h^0(V,\mcO_V(a_0 A_0+ \dots + a_r A_r))}{k^d/d!} \nonumber \\
=\frac{d!}{v!}\limsup_k &\sum_{\substack{\lambda_0+\dots+\lambda_r=1\\ \lambda_i \in \frac{1}{k} \BN}} \frac{h^0(V,\mcO_V(k(\lambda_0 A_0+ \dots + \lambda_r A_r)))}{k^r k^v/v!} \label{zwischen} \ .
\end{align}
The Fujita approximation theorem \cite{Fujita} shows that the $\limsup$ on the right hand side of \eqref{zwischen} can be changed to a $\lim$. Let $\Delta$ denote the standard $r$-simplex. The right hand side can then
be regarded as the integral over $\Delta$ of a suitable step function $\chi_k$ with steps of ($r$-dimensional) volume $1/k^r$, and $\chi_k$ converges pointwise to
$\vol_V(\lambda_0 A_0 + \dots + \lambda_r A_r)$.

In order to complete the proof using the dominated convergence theorem, it suffices to show that all $\chi_k$ are bounded by an integrable function. For this we use \cite[Proposition 3.5.1]{ArchV} which implies that 
\[ \left| \frac{h^0(V,\mcO_V(k(\lambda_0 A_0+ \dots + \lambda_r A_r)))}{k^v/v!} - \vol_V(\mcO_V((\lambda_0 A_0+ \dots + \lambda_r A_r))) \right|    \]
is bounded in $k$ and since $\vol_V$ is bounded on $\Delta$ all $\chi_k$ are bounded by a constant over $\Delta$.
\end{proof}

\begin{rem}
Let $E$ be an elliptic curve without complex multiplication and take $V=E \times E$. The
nef cone of $V$ is equal to its effective cone and it is circular. Taking suitable Cartier Divisors $A_0$ and $A_1$ with $A_0$ ample and
$A_1$ not nef, Cutkosky utilizes the geometry of $\text{Nef}(V)$ to deduce that $\vol_X(\mcO_X(1))$ is a quadratic irrationality. 
\end{rem}

\subsection{Abelian Varieties}

Let $(A,L_0)$ be a polarized $d$-dimensional complex abelian variety, then 
each $L \in N^1(A)$ induces a homomorphism $\phi_{L}:A \rightarrow \hat{A}$, where $\hat{A}$ is the dual abelian variety of $A$. The map $\phi_{L_0}$ is an
isogeny, and we obtain an isomorphism of
$\BQ$-vector spaces
\begin{align} 
\begin{split}
	\varphi: N^1_\BQ(A) &\rightarrow \End^s_\BQ(A) \label{Iso} \\
L &\mapsto \phi_{L_0}^{-1} \phi_L \,, 
\end{split}
\end{align}
where $\End^s_\BQ(X)$ denotes the subspace of $\End_\BQ(A)$ fixed by the Rosati involution with respect to the polarization $L_0$.\\

\cite[Proposition 5.2.3]{Birkenhake04} shows that the characteristic polynomial $P_{f_L}^a$ of the analytic representation $f_L$ of $\phi_{L_0}^{-1} \phi_L \in \End^s_\BQ(A)$ satisfies 
\[ P_{f_L}^a(t)=\frac{(t L_0-L)^d}{d_0 d!}\,, \]
where $d_0$ denotes the degree of the polarization $L_0$. \\

From now on we concentrate on abelian varieties with real multiplication. 

\begin{lemma}
	\label{Hlemma}
	Let $K$ be a totally real number field of degree $d$ over $\BQ$ with primitive element $\alpha$. Then there exists a $d$-dimensional polarized simple abelian variety $(A,L_0)$ and a line bundle $L$ on $A$ such that the volume function $\vol_A(t L_0-L)$ restricted to the nef cone of $A$ is given by a rational multiple of the minimal polynomial of $\alpha$.    
\end{lemma}

\begin{proof}
It is possible to construct a simple abelian polarized abelian variety with $\End_\BQ(A)=K$ \cite[Proposition 9.2.1]{Birkenhake04}. We also have $\End^s_\BQ(A)=\End_\BQ(A)$ since $K$ is totally real \cite[Proposition 5.5.7]{Birkenhake04}.

Take $(A,L_0)$ with this property and let $L$ be $\varphi^{-1}(\alpha)$. The minimal polynomial of $\alpha$ is equal to the characteristic polynomial of the analytic representation of the corresponding endomorphism because they both have degree $d$ and vanish at $\alpha$. By the above discussion and the Riemann--Roch theorem the lemma is proven.  
\end{proof}

In the case of abelian varieties the boundary of the nef cone has the following property \cite[Corollary 1.5.18]{PAG}.

\begin{lemma}
	Let $A$ be an abelian variety of dimension $d$ and $\delta \in N_\BR^1(A)$ a numerical equivalence class on $A$ which lies in the boundary of the nef cone. Then \[(\delta^d)=0.\] 
\end{lemma}

Furthermore abelian varieties have the useful property that the nef cone and the pseudoeffective cone coincide. This is even true in a more general setting \cite[Example 1.4.7]{PAG}.

\begin{lemma}
	Let $V$ be a complete variety with a connected algebraic group acting transitively on it.
	Then every effective divisor on $V$ is nef.  
\end{lemma} 	

\section{Algebraic Volumes of Divisors}
\noindent We now proceed with the proof of the main theorem.

\begin{proof}[Proof of Theorem \ref{thm:haupt}]
Let $\alpha \in K$ be a primitive element. We use Lemma \ref{Hlemma} to find a polarized abelian variety $(A,L_0)$ with the property that $\vol_A(t L_0-L)$ restricted to the nef cone of $A$ is given by a rational multiple of the minimal polynomial of $\alpha$, where $L$ corresponds to $\alpha$ via the isomorphism \eqref{Iso}.
We choose $t_0 \in \BN$ large enough so that $t_0 L_0-L$ is ample and $-t_0 L_0-L$ is not pseudoeffective.
	
Considering Cutkosky's construction with $r=1, A_0 = -t_0 L_0-L$, and $A_1 = t_0 L_0-L$ so that \[X=\BP(\mcO_A(-t_0 L_0-L) \oplus \mcO_A(t_0 L_0-L)) \,,\]
we will compute $\vol_X(\mcO_X(1))$.
	
Lemma \ref{int} yields that $\vol_X(\mcO_X(1))$ is a rational
multiple of the integral 
\[ \int_{-t_0}^{t_0} \vol_A(t L_0 - L) \, \mathrm{d}t \,, \]
while Lemma \ref{Hlemma} shows that the latter is a rational
multiple of \[ \int_\beta^{t_0} m_\alpha(t) \, \mathrm{d}t   \]
with $m_\alpha$ the minimal polynomial of $\alpha$ over $\BQ$, and $\beta$ the largest root of $m_\alpha$ (which is still a primitive element of $K$ since $K$ is galois). 
This number is not necessarily a primitive element of $K$, but we will show that primitivity holds for a general choice of
$\alpha$. This will be established in the following lemmata. \\ 
Note also that the construction only works if $L$ is an integral Cartier divisor. But this is no problem since for $k \in \BN$ we have $m_{k \alpha}(t)=k^{[K:\BQ]} m_\alpha(t/k)$ and therefore \[ \int_{k \beta}^{k t_0} m_\alpha(t) \, \mathrm{d}t = \int_\beta^{t_0} m_\alpha(t) \, \mathrm{d}t \,. \]
\end{proof}

\begin{lemma}
\label{generic}
Let $K$ be a number field (not necessarily totally real) with $[K:\BQ]=d$. There exists a nonzero polynomial $G \in \BQ[X_0,\dots,X_{d-1}]$ with the following property: if a primitive element $\alpha \in K$ has minimal polynomial $m_\alpha = a_0 + a_1 X + \dots + a_{d-1} X^{d-1} + X^d$ then the nonvanishing of $G(a_0,\dots,a_{d-1})$ is equivalent to 
$M_\alpha(\alpha)$ being a primitive element for $K$, where $M_\alpha$ denotes the antiderivative of $m_\alpha$ satisfying $M_\alpha(0)=0$.
\end{lemma}

\begin{proof}
Since $\alpha$ is a primitive element of $K$, the powers  of $\alpha$ constitute a basis of $K$. Therefore $M_\alpha(\alpha)$ is a primitive element for $K$ if the matrix $A$, whose $n$-th column consists of the coordinates of $M_\alpha(\alpha)^n$ with respect to this basis, has nonzero determinant.

The expression $M_\alpha(\alpha)^n$ is a polynomial in the coefficients $a_0,\dots,a_{d-1}$ of $m_\alpha$ and we can use $m_\alpha(\alpha)=0$ to eliminate powers of $\alpha$ with degree greater or equal $d$. This shows that $\det(A)$ is given by a polynomial $G$ in $a_0,\dots,a_{d-1}$ with the required property. Note that this polynomial depends only on $d$ and not on the number field $K$.

Suppose $G$ is the zero polynomial. Then $M_\alpha(\alpha)$ is never a primitive element for \emph{any} number field $K$ of degree $d$. But this is a contradiction to the case $K=\BQ(\sqrt[d]{2})$, where $M_\alpha(\alpha)$ is easily seen to be a primitve element for $\alpha=\sqrt[d]{2}$.       
\end{proof}

We now need to prove that the polynomial constructed in Lemma \ref{generic} does not vanish for every primitive element of the number field in question. 
 
\begin{prop}
\label{nonzero}
Let $K$ be a totally real number field with $[K:\BQ]=d$, and let $P \in \BQ[X_0,\dots,X_{d-1}]$ be a polynomial with the property that $P(a_0,\dots,a_{d-1})=0$ for the coefficients $a_0,\dots,a_{d-1}$ of the minimal polynomial of every primitive element of $K$. Then $P$ is the zero polynomial.    
\end{prop}

\begin{proof}
The weak approximation theorem \cite[Theorem 3.4]{Neukirch} shows that the image of the map \[ K \rightarrow \BR^{d} \] which maps an element of $K$ to its conjugates is dense with respect to the usual topology.

An element $\alpha \in K$ is primitive if and only if the components of the image of $\alpha$ in $\BR^{d}$ are pairwise different. This means that the image of all primitive elements in $K$ will be the subset of the image of $K$ whose elements have pairwise different components and this is clearly still dense in $\BR^{d}$. 

Now assume there is a polynomial $P$ with the asserted property. Then the relation $P(a_0,\dots,a_{d-1})=0$ for the coefficients of the minimal polynomial of every primitive element of $K$ implies a relation between the conjugates of every primitive element and this relation is nontrivial if and only if $P$ is nonzero because of the fundamental theorem of symmetric polynomials. Now $P$ has to be zero because its zero set in $\BR^{r_1} \times \BC^{r_2}$ contains a dense subset in the standard topology.
\end{proof}

The following lemma then concludes the proof of Theorem \ref{thm:haupt}.

\begin{lemma}
Let $K$ be a totally real number field with $[K:\BQ]=d$ and let \[S:=\{\alpha \in K \mid \text{$\alpha$ and $M_\alpha(\alpha)$ primitive elements of $K$}\}\]  where $M_\alpha$ denotes the antiderivative of the minimal polynomial of $\alpha$ with $M_\alpha(0)=0$. For a fixed isomorphism $K \rightarrow \BQ^d$ the image of $S$ is a nonempty Zariski open subset of $\BQ^d$.  
\end{lemma}

\begin{proof}
Let $v_1, \dots, v_d$ be a $\BQ$-Basis of $K$. The same argument as in Lemma \ref{generic} shows that for $\lambda_1,\dots,\lambda_d \in \BQ$ the linear combination $x=\lambda_1 v_1+\dots+\lambda_d v_d$ is a primitive element of $K$ if and only if $H(\lambda_1,\dots,\lambda_d) \neq 0$ for a certain nonzero $H \in \BQ[X_1,\dots,X_d]$. By Lemma \ref{generic} and Proposition \ref{nonzero} and the fact that whenever $\lambda_1 v_1+\dots+\lambda_d v_d$ is a primitive element, the coefficients of its minimal polynomial are polynomials in $\lambda_1,\dots,\lambda_d$, we conclude that there is a nonzero $B \in \BQ[X_1,\dots,X_d]$ with the property that for $B(\lambda_1,\dots,\lambda_d) \neq 0$ we have $M_x(x)$ is primitive.

The intersection of the complements of the zero sets of $H$ and $B$ then has the asserted properties, in particular it is nonempty because $\BQ^d$ is irreducible with respect to the Zariski topology.         
\end{proof}

Our results raise the following question:

\begin{question}
Is it possible to extend this method to show that every nonnegative (totally) real algebraic number appears as $\vol_X(D)$ for some pair $(X,D)$?
\end{question}

In the construction we used an algebraic variety of dimension $d+1$ to obtain an algebraic volume of degree $d$. Consequently we ask the following question. 
\begin{question}
\label{CD}
Is there a constant $C_d$ such that if $\vol_X(D)$ is algebraic over $\BQ$, then we have $\deg_\BQ \vol_X(D) \leq C_d$ for every algebraic variety $X$ of dimension $d$ and every Cartier divisor $D$ on $X$.
\end{question}

We show that the bound $C_d$ can not be linear in $d$.

\begin{lemma}
For every pair of odd prime numbers $p<q$ with $p \nmid q-1$ there exists a variety $V$ of dimension $p+q+2$ and a Cartier divisor $D$ on $V$ such that $\vol_V(D)$ is algebraic over $\BQ$ of degree $pq$.    
\end{lemma}

\begin{proof}
Given $p$ and $q$, we use Theorem $\ref{thm:haupt}$ to find a pair $(X,M)$ with $\dim(X)=p+1$, and a pair $(Y,E)$ with $\dim(Y)=q+1$ and the additional property that $\BQ(\vol_X(M))$ and $\BQ(\vol_Y(E))$ are Galois of degree $p$ and $q$ respectively. This implies that $\BQ(\vol_X(M),\vol_Y(E))$ is Galois with Galois group $\BZ/pq\BZ$.

Because of this the only nontrivial intermediate fields are $\BQ(\vol_X(M))$ and $\BQ(\vol_Y(E))$ and therefore $\vol_X(M) \cdot \vol_Y(E)$ has to be a primitive element of $\BQ(\vol_X(M),\vol_Y(E))$.

We now take $V= X \times Y$ and $D=p_1^* M \otimes p_2^* E$, where $p_i$ is the $i$-th projection, and obtain the assertion from the Künneth formula.  
\end{proof}

\section{Realizing $\pi$ as volume of a divisor}
\noindent In this section we show how Cutkosky's construction can be used to realize $\pi$ as the volume of a divisor on a projective bundle. The idea is to utilize the shape of the nef cone of the product of an elliptic curve with itself. This was already used in \cite{Volumefunctions} to show that there exist divisors with transcendental volume.\\

Our exposition follows \cite{Volumefunctions}. Let $E$ be an elliptic curve without complex multiplication and set $V=E \times E$. A complete description of the nef and pseudoeffective cone is given by \cite[Lemma 1.5.4]{PAG}.

The nef cone is equal to the pseudoeffective cone and the dimension of the N\'eron--Severi space is 3. The Hodge index theorem then implies that the pseudoeffective cone is circular and that we can choose an orthogonal basis $A, D_1, D_2$ of $N^1_\BR(X)$ such that $A$ is integral and ample and $D_1, D_2$ are integral with \[A^2=-D_1^2=-D_2^2=N\] for some $N \in \BN$. 

\begin{prop}
On $V$ take \[M_0 = A-2 D_1 - 2 D_2,\, M_1 = A + 2 D_1 -2 D_2,\, M_2 = A + 4 D_1\,,\] and let $X=\BP(\mcO_V(M_0) \oplus \mcO_V(M_1) \oplus \mcO_V(M_2))$. Then the volume of $\mcO_X(1)$ is a rational multiple of $\pi$. 
\end{prop}

\begin{proof}
We use Lemma \ref{int} to obtain 
\[ \vol_X(\mcO_X(1))= 12 \int_{\substack{\lambda_0+\lambda_1+\lambda_2=1\\\lambda_i \geq 0}} \vol_V(\lambda_0 M_0 + \lambda_1 M_1 + \lambda_2 M_2)\, \md \lambda_1 \wedge \md \lambda_2 \,. \] 
The volume function vanishes outside of the pseudoeffective cone, and it is easy to check that the intersection of the region \[\{\lambda_0 M_0+\lambda_1 M_1+\lambda_2 M_2 \mid \lambda_0+\lambda_1+\lambda_2=1 ,  \lambda_i \geq 0\}\] and the pseudoeffective cone is the disk given by $\{ A +x D_1 +y D_2 \mid x^2+y^2 \leq 1 \}$. This shows that 
\[ \vol_X(\mcO_X(1)) = 12 \int_{x^2+y^2 \leq 1} \vol_V(A +x D_1 +y D_2) \md x \wedge \md y \,.\] 

The asymptotic Riemann--Roch theorem along with the fact that the pseudo-effective and nef cones of $E \times E$ coincide lead to
\begin{align*} \vol_X(\mcO_X(1)) &= 6N \int_{x^2+y^2 \leq 1} 1-x^2-y^2 \,\md x \wedge \md y \\&= 6\pi N  - 6N \int_{x^2+y^2 \leq 1} x^2+y^2 \,\md x \wedge \md y \,. \end{align*}
Using polar coordinates we obtain
\[ \vol_X(\mcO_X(1)) = 6\pi N - 6 N \int_0^{2\pi} \int_0^1 r^3 \, \md r \, \md \varphi \ , \]
yielding \[ \vol_X(\mcO_X(1)) = 3 \pi N \ . \]
\end{proof}

Since the set of volumes is a multiplicative semigroup containing $\BQ_+$ we obtain the following corollary.

\begin{cor}
$\pi$ is an element of $\BV$.
\end{cor}

\begin{rem}
Every nonnegative rational number arises as the volume of a divisor on a surface \cite[Proposition 2.46]{Hab}. Using the Künneth formula $\pi$ is then obtained as volume of a divisor on a variety of dimension 6.
\end{rem}

It is natural to ask which is the smallest dimension where a given $\alpha \in \BR^+$ can be realized as a volume. 

\begin{defin}
For any real number $\alpha \in \BV$ we define \[ \nu(\alpha) := \min \left\{ \dim(X) \mid \vol_X(D)=\alpha \text{ for some Cartier divisor $D$ on $X$}  \right\}.\] 
\end{defin}

\begin{rem}
It is known that $\nu(n) = 1$ for $n \in \BN$ and $\nu(q) = 2$ for $q \in \BQ^+ \setminus \BN$. The remark above shows that $\nu(\pi) \leq 6$. It seems to be very interesting, if this bound can be improved to compute $\nu$ for other elements of $\BV$. Also note that a positive answer to Question \ref{CD} would imply that $\nu(\alpha) > d$ for every algebraic $\alpha$ with $\deg(\alpha) > C_d$. 
\end{rem}

\bibliographystyle{amsalpha}
\bibliography{Article}

\end{document}